\documentclass[11pt,reqno,british,twoside]{article}

\usepackage[utf8]{inputenc}
\usepackage[sc]{mathpazo}
\usepackage{babel,mathtools,amsmath,amssymb,url,paralist,xspace,mathrsfs}
\usepackage{fancyhdr}
\usepackage[margin=30pt,font=small,labelfont=bf,labelsep=period]{caption}
\usepackage[amsmath,thmmarks]{ntheorem}
\rmfamily
\mathtoolsset{mathic=true}
\setdefaultenum{\upshape (i)}{}{}{}

\makeatletter
\DeclareRobustCommand*{\bfseries}{%
  \not@math@alphabet\bfseries\mathbf
  \fontseries\bfdefault\selectfont
  \boldmath
}
\makeatother

\theoremstyle{plain}
\theoremseparator{.}
\newtheorem{theorem}{Theorem}[section]
\newtheorem{proposition}[theorem]{Proposition}

\theorembodyfont{\normalfont}
\newtheorem{definition}[theorem]{Definition}

\theoremheaderfont{\itshape}
\theoremsymbol{\begin{small}\ensuremath{\quad\triangle}\end{small}}
\newtheorem{remark}[theorem]{Remark}
\theoremsymbol{\begin{small}\ensuremath{\quad\diamondsuit}\end{small}}
\newtheorem{example}[theorem]{Example}

\theoremstyle{nonumberplain}
\theoremheaderfont{\itshape}
\theorembodyfont{\normalfont}
\theoremsymbol{\begin{small}\ensuremath{\quad\Box}\end{small}}
\newtheorem{proof}{Proof}

\qedsymbol{\begin{small}\ensuremath{\quad\Box}\end{small}}

\numberwithin{equation}{section}


\let\oldbibliography\thebibliography
\renewcommand{\thebibliography}[1]{%
  \oldbibliography{#1}%
  \small
  \setlength{\itemsep}{0pt}%
  \setlength{\parskip}{0pt}%
}

\newcommand{\Lie}[1]{\operatorname{\textsl{#1}}}

\newcommand{\GL}{\Lie{GL}}

\newcommand{\SU}{\Lie{SU}}
\newcommand{\Spin}{\Lie{Spin}}
\newcommand{\Ld}{\mathcal L}
\newcommand{\CST}{\mathscr C}
\newcommand{\sgn}{sgn}

\newcommand{\Hodge}{{*}}

\newcommand{\bC}{{\mathbb C}}
\newcommand{\bR}{{\mathbb R}}
\newcommand{\bZ}{{\mathbb Z}}

\DeclareMathOperator{\im}{Im}
\DeclareMathOperator{\re}{Re}
\DeclareMathOperator{\Tr}{Tr}
\DeclareMathOperator{\diag}{diag}
\DeclareMathOperator{\rank}{rank}

\DeclareMathOperator{\vol}{vol}

\newcommand{\hook}{{\lrcorner\,}}
\newcommand{\NB}{\nabla}
\newcommand{\LC}{\NB^{\textup{LC}}}

\DeclarePairedDelimiter{\norm}{\lVert}{\rVert}
\DeclarePairedDelimiter{\Span}{\langle}{\rangle}
\DeclarePairedDelimiter{\base}{\lparen}{\rparen}

\renewcommand{\bfdefault}{b}

\newcommand{\br}{\hspace{0pt}}
\newcommand{\bdash}{-\br} %dash allowing hyphen in next word
\newcommand{\eqbreak}{\\&\qquad}

\begin{document}
                                            
\thispagestyle{empty}
\begin{small}
  \begin{flushright}
    IMADA Preprint 2011\\
    CP\textsuperscript3-ORIGINS: 2011-9
  \end{flushright}
\end{small}

\bigskip
 
\begin{center}
  \LARGE\bfseries Spin(7)-manifolds with \\
  three-torus symmetry
\end{center}
\begin{center}
  \Large Thomas Bruun Madsen
\end{center}

\begin{abstract}
  Metrics of exceptional holonomy are vacuum solutions to the Einstein
  equation. In this paper we describe manifolds with holonomy
  contained in \( \Spin(7) \) preserved by a three-torus symmetry in
  terms of tri-symplectic geometry of four-manifolds. These complement
  examples that have appeared in the context of domain wall problems
  in supergravity.
\end{abstract}

\let\thefootnote\relax\footnotetext{2010 Mathematics Subject
  Classification: Primary 53C15; Secondary 53C29, 53D20, 70G45}

\section{Introduction}
\label{sec:introduction}
Metrics of exceptional holonomy have received much attention from both
mathematicians and physicists over the years. The mathematical
motivation for studying exceptional holonomy metrics began with
Berger's classification of Riemannian holonomy groups
\cite{Berger:hol}, though their existence was first shown much later
in Bryant's paper \cite{Bryant:exceptional}. Significant results then
followed, in particular it is worth mentioning the complete
exceptional holonomy metrics discovered by Bryant and Salamon
\cite{Bryant-Salamon:exceptional} and Joyce's construction
\cite{Joyce:G2-1-2,Joyce:Spin7} of compact Riemannian manifolds with
holonomy \( G_2 \) and \( \Spin(7) \). In this paper we focus on
holonomy \( \Spin(7) \)-metrics. From the physical perspective one
motivation for studying these metrics comes from superstring theories
\cite{Acharya-G:M,Cvetic-GLP:Spin7-2,Cvetic-GLP:Spin7-4,Gukov-S:Spin7,Santillan-S:Sp7G2}.
Recently, we \cite{Madsen-S:multi-moment} used the notion of
multi-moment map to study torsion-free \( G_2 \)-metrics admitting an
isometric action of \( T^2 \). In this paper we use similar ideas to
study holonomy \( \Spin(7) \)-metrics with \( T^3 \) symmetry;
symmetry groups of \( \rank \) three fit well with Reidegeld's study
\cite{Reidegeld:cohom1} of cohomogeneity one \( \Spin(7) \)-metrics.

The paper is organised as follows. In section \ref{sec:Sp7red} we
briefly explain the notion of multi-moment maps for geometries with a
closed four-form and \( T^3 \) symmetry. We then study how to reduce a
torsion-free \( \Spin(7) \)-manifold to a tri-symplectic four-manifold
and thereafter, in section \ref{sec:Sp7inv}, explain how to obtain all
torsion-free \( \Spin(7) \)-manifolds with free \( T^3 \) symmetry
starting from tri-symplectic four-manifolds. In the final section of
the paper we present two examples illustrating our reduction and
reconstruction procedures. One of the examples complements previous
ones that have appeared in the context of domain wall problems in
supergravity theories
\cite{Gibbons-LPS:domain-walls,Santillan-Gi:toric}.

\paragraph*{Acknowledgements} It is a pleasure to thank Andrew Swann
for enlightening discussions and Bent {\O}rsted for useful comments. I
also acknowledge financial support from \textsc{ctqm},
\textsc{geomaps} and \textsc{opalgtopgeo}.

\section{Reduction via multi-moment maps}
\label{sec:Sp7red}

In \cite{Madsen-S:multi-moment} we developed a notion of multi-moment
map for geometries with a closed three-form. We explain in
\cite{Madsen:phdtorsion} how this idea generalises to higher degree
forms. For a manifold \( Y \) endowed with a closed four-form \( \Phi
\) and an action of a three-torus preserving \( \Phi \) the definition
is particularly simple. A \emph{multi-moment map}, for \( T^3 \)
acting on \( (Y,\Phi) \), is an invariant function \(
\nu\colon\,Y\to\bR \) such that
\begin{equation}
  \label{eq:mmmaprel}
  d\nu=\Phi(U_1,U_2,U_3,\cdot),
\end{equation}
where the vector fields \( U_1 \), \( U_2 \) and \( U_3 \) generate
the \( T^3 \) action. Following \cite[Theorem
3.1(i)]{Madsen-S:multi-moment} one finds that such a multi-moment map
is guaranteed to exist provided that \( b_1(Y)=0 \).

\begin{remark}
  Note that closedness of the one-form \( U_3\hook U_2\hook
  U_1\hook\Phi \) follows by applying Cartan's formula: \(
  0=\Ld_{U_1}\Phi=d(U_1\hook\Phi) \), \(
  0=\Ld_{U_2}(U_1\hook\Phi)=d(U_2\hook U_1\hook\Phi) \), \(
  0=\Ld_{U_3}(U_2\hook U_1\hook\Phi)=d(U_3\hook U_2\hook U_1\hook\Phi)
  \).
\end{remark}

Let us now recall the fundamental aspects of \( \Spin(7) \)-geometry
following \cite{Bryant:exceptional}. On \( \bR^8 \) we consider the
four-form \( \Phi_0 \) given by
\begin{equation}
  \label{eq:Phi0}
  \begin{split}
    \Phi_0 = e_{1234} & + (e_{12}+e_{34})(e_{56}+e_{78})+(e_{13}-e_{24})(e_{57}-e_{68})\\
    & -(e_{14}+e_{23})(e_{58}+e_{67})+e_{5678},
  \end{split}
\end{equation}
where \( e_1,\dots,e_8 \) is the standard dual basis and wedge signs
have been omitted. The stabiliser of \( \Phi_0 \) is the compact \( 21
\)-dimensional Lie group
\begin{equation*}
  \Spin(7) = \{\, g\in\GL(8,\bR) : g^*\Phi_0=\Phi_0\,\}.
\end{equation*}
This group preserves the standard metric \( g_0 = \sum_{i=1}^8{e_i}^2
\) on \( \bR^8 \) and the volume form \( \vol_0 = e_{12345678} \).
These tensors are uniquely determined by \( \Phi_0 \) via the
relations \( 14\vol_0=\Phi_0^2 \) and \( (Y\hook X\hook \Phi_0)\wedge
(Y\hook X\hook \Phi_0)\wedge \Phi_0 = 6 \norm{X \wedge Y }^2 \vol_0
\), cf. \cite{Kar-G2Sp7}. The form \( \Phi_0 \) is self-dual, meaning
\( \Hodge\Phi_0 =\Phi_0 \).

A \(\Spin(7) \)-structure on an eight-manifold \( Y \) is given by a
four-form \( \Phi\in\Omega^4(Y) \) which is linearly equivalent at
each point to \( \Phi_0 \). It determines a metric \( g \) and a
volume form \( \vol \). The \( \Spin(7) \)-structure is called
\emph{torsion-free} if the form \( \Phi \) is parallel with respect to
the Levi-Civita connection, meaning \( \LC \Phi = 0 \). This happens
precisely when \( \Phi \) is closed. One then calls \( (Y,\Phi) \) a
torsion-free \( \Spin(7) \)-manifold. In this situation the metric \(
g \) has holonomy contained in \( \Spin(7) \) and is Ricci-flat. In
particular, \( g \) is real-analytic in harmonic coordinates.

Since a torsion-free \( \Spin(7) \)-manifold comes equipped with a
closed four-form, we may study multi-moment maps for such manifolds.
Assume that \( (Y,\Phi) \) has a three-torus symmetry, generated by
vector fields \( U_i \), necessarily real-analytic
\cite[Theorem~2.3]{Kob:transgrp}, and that there is a non-constant
multi-moment map \( \nu \). Then \( d\nu = \Phi(U_1,U_2,U_3,\cdot) \)
is non-zero if and only if \( U_1 \), \( U_2 \) and \( U_3 \) are
linearly independent, cf. \cite{Fern:Spin7}.  So \( T^3 \) acts
locally freely on some open set \( Y_0 \subset Y \).

Let us define three two-forms on \( Y_0 \) by
\begin{equation*}
  \omega_1=U_2\hook U_3\hook \Phi, \quad \omega_2=U_3\hook U_1\hook \Phi,\quad \omega_3 = U_1\hook U_2\hook \Phi.
\end{equation*}
To relate these to the \( \Spin(7) \)-structure we introduce two \(
\bR^3 \)-valued one-forms \( \theta=(\theta_1,\theta_2,\theta_3) \)
and \( \Theta=(\Theta_1,\Theta_2,\Theta_3) \). The one-form \( \theta
\) is defined by the formula \( \theta=U^\flat G^{-1} \), where \(
U^\flat \) has entries \( U_i^\flat=g(U_i,\cdot) \), and \(
G^{-1}=(g^{ij}) \) denotes the inverse of the matrix \( G=(g_{ij}) \)
that has entries \( g_{ij}=g(U_i,U_j) \). Note that \(
\theta_i(U_j)=\delta_{ij} \). The second \( \bR^3 \)-valued one-form
is given by the formula \( \Theta = h^2U^\flat \), where \( h \) is
the positive real-analytic function \( h=\sqrt{\det(G^{-1})} \);
componentwise we have \( \Theta_i=h^2\sum_{j=1}^3g_{ij}\theta_j \).

\begin{proposition}
  \label{prop:Phi}
  On \( Y_0 \), the four-form \( \Phi \) is
  \begin{equation}
    \label{eq:Sp7form}
    \begin{split}
      \Phi&=d\nu\wedge\left(2\,\theta_2\wedge\theta_3\wedge\theta_1+\Theta_1\wedge\omega_1+\Theta_2\wedge\omega_2+\Theta_3\wedge\omega_3\right)\eqbreak
      +\theta_3\wedge\theta_2\wedge\omega_1+\theta_1\wedge\theta_3\wedge\omega_2+\theta_2\wedge\theta_1\wedge\omega_3
      + \Hodge(d\nu\wedge\theta_3\wedge\theta_2\wedge\theta_1).
    \end{split}
  \end{equation}
\end{proposition}
\begin{proof}
  Working locally at a point and using the \( T^3 \)-action we may
  write the first three standard basis elements of \( \bR^8 \) as \(
  E_1=k_1U_1 \), \( E_2= k_2U_1+\ell_2U_2 \), \( E_3=
  k_3U_1+\ell_3U_2+m_3U_3 \) for appropriate functions \(
  k_1,\ldots,m_3 \). Now, using \eqref{eq:Phi0}, we get \(
  k_1\ell_2\,\omega_3 = -e_{34}-e_{56}-e_{78} \), \(
  k_1m_3\,\omega_2-k_1\ell_3\,\omega_3=-e_{24}+e_{57}-e_{68} \) and \(
  -\ell_2m_3\,\omega_1+k_2m_3\,\omega_2+(\ell_2k_3-k_2\ell_3)\,\omega_3=e_{14}-e_{58}-e_{67}
  \). We therefore have
  \begin{gather*}
    \ell_2m_3\,\omega_1=-e_{14}+e_{58}+e_{67}-\tfrac{k_2}{k_1}(e_{24}-e_{57}+e_{68})-\tfrac{k_3}{k_1}(e_{34}+e_{56}+e_{78})\\
    k_1m_3\,\omega_2=-e_{24}+e_{57}-e_{68}-\tfrac{\ell_3}{\ell_2}(e_{34}+e_{56}+e_{78})\\
    k_1\ell_2\,\omega_3=-e_{34}-e_{56}-e_{78}.
  \end{gather*}
  Next, we write \( \theta_1 = k_1e_1+k_2e_2+k_3e_3 \), \( \theta_2 =
  \ell_2e_2+\ell_3e_3 \) and \( \theta_3=m_3e_3 \). Also note that \(
  h\,d\nu=e_4 \). We then find
  \begin{equation*}
    \begin{split}
      e_{1234}&=d\nu\wedge\theta_3\wedge\theta_2\wedge\theta_1,\quad e_{5678}=\Hodge(d\nu\wedge\theta_3\wedge\theta_2\wedge\theta_1),\\
      \theta_3\wedge\theta_2\wedge\omega_1&=e_{1234}-e_{23}(e_{58}+e_{67})-\tfrac{k_2}{k_1}e_{23}(e_{57}-e_{68})+\tfrac{k_3}{k_1}e_{23}(e_{56}+e_{78}),\\
      \theta_1\wedge\theta_3\wedge\omega_2&=e_{1234}+e_{13}(e_{57}-e_{68})-\tfrac{\ell_3}{\ell_2}e_{13}(e_{56}+e_{78})\eqbreak+\tfrac{k_2}{k_1}e_{23}(e_{57}-e_{68})-\tfrac{k_2\ell_3}{k_1\ell_2}e_{23}(e_{56}+e_{78}),\\
      \theta_2\wedge\theta_1\wedge\omega_3&=e_{1234}+e_{12}(e_{56}+e_{78})-\tfrac{k_3}{k_1}e_{23}(e_{56}+e_{78})\eqbreak+\tfrac{k_2\ell_3}{k_1\ell_2}e_{23}(e_{56}+e_{78})+\tfrac{\ell_3}{\ell_2}e_{13}(e_{56}+e_{78}),\\
      d\nu\wedge(\Theta_1\wedge\omega_1&+\Theta_2\wedge\omega_2+\Theta_3\wedge\omega_3)=-e_{14}(e_{58}+e_{67})-e_{24}(e_{57}-e_{68})\eqbreak\qquad\qquad\qquad\qquad+e_{34}(e_{56}+e_{78}),
    \end{split}
  \end{equation*}
  and the given expression for \( \Phi \) follows.
\end{proof}

\begin{remark}
  The functions \( k_1,\ldots,m_3 \) from the proof of Proposition
  \ref{prop:Phi} are related to \( G \) in the following way
  \begin{equation*}
    G=\left(\begin{smallmatrix} \tfrac1{k_1^2} & -\tfrac{k_2}{k_1^2\ell_2} & \tfrac{k_2\ell_3-k_3\ell_2}{k_1^2\ell_2m_3}\\
        -\tfrac{k_2}{k_1^2\ell_2} & \tfrac{k_2^2}{k_1^2\ell_2^2}+\tfrac{1}{\ell_2^2} & \tfrac{k_2(k_3\ell_2-k_2\ell_3)}{k_1^2\ell_2^2m_3}-\tfrac{\ell_3}{\ell_2^2m_3}\\
        \tfrac{k_2\ell_3-k_3\ell_2}{k_1^2\ell_2m_3} & \tfrac{k_2(k_3\ell_2-k_2\ell_3)}{k_1^2\ell_2^2m_3}-\tfrac{\ell_3}{\ell_2^2m_3} & \tfrac{(k_2\ell_3-k_3\ell_2)^2}{(k_1\ell_2m_3)^2}+\tfrac{\ell_3^2}{\ell_2^2m_3^2}+\tfrac1{m_3^2}   
      \end{smallmatrix}\right),
  \end{equation*}
  and for \( G^{-1}=(g^{ij}) \) we have
  \begin{equation}
    \label{eq:Ginv}
    G^{-1} = \left(\begin{smallmatrix}k_1^2+k_2^2+k_3^2 & k_2\ell_2+k_3\ell_3 & k_3m_3\\
        k_2\ell_2+k_3\ell_3 & \ell_2^2+\ell_3^2 & \ell_3m_3\\
        k_3m_3 & \ell_3m_3 & m_3^2
      \end{smallmatrix}\right).
  \end{equation}
\end{remark}

Now suppose that \( t\in\nu(Y_0) \) is a regular value for \(
\nu\colon\,Y_0\to\bR \). Then \( \mathcal X_t=\nu^{-1}(t) \) is a
real-analytic hypersurface and has unit normal \( N=h(d\nu)^\sharp \).
We shall denote by \( \iota \) the inclusion \( \mathcal
X_t\hookrightarrow Y_0 \).

\begin{definition}
  The \emph{\( T^3 \) reduction} of \( Y_0 \) at level \( t \) is the
  four-manifold
  \begin{equation*}
    M=\nu^{-1}(t)/{T^3}=\mathcal X_t/{T^3}.
  \end{equation*}
\end{definition}
This quotient space is a tri-symplectic manifold.
\begin{proposition}
  \label{prop:sympltripl}
  The \( T^3 \) reduction \( M \) carries three pointwise linearly
  independent symplectic forms defining the same orientation.
\end{proposition}
\begin{proof}
  Consider the real-analytic two-forms \( \omega_1 \), \( \omega_2 \) and \(
  \omega_3 \) on \( Y_0 \). These forms are \( T^3 \)-invariant and
  closed since for instance \( \mathcal L_{U_i}\omega_1=\mathcal
  L_{U_i}(U_2\hook U_3\hook\Phi)=0 \) and \( d\omega_1=d(U_2 \hook
  U_3\hook \Phi)=\mathcal L_{U_2}(U_3\hook\Phi)=0 \), respectively.
  Furthermore, as \( U_1\hook\omega_1=-d\nu \), etc., their pull-backs
  to \( \mathcal X_t = \nu^{-1}(t) \) are basic. Thus they descend to
  three closed forms \( \sigma_1 \), \( \sigma_2 \) and \( \sigma_3 \)
  on \( M \).

  The proof of Proposition \ref{prop:Phi} shows that at a point \(
  k_1\ell_2m_3\,\sigma_1=k_1(e_{58}+e_{67})+k_2(e_{57}-e_{68})-k_3(e_{56}+e_{78})
  \), \(
  k_1\ell_2m_3\,\sigma_2=\ell_2(e_{57}-e_{68})-\ell_3(e_{56}+e_{78})
  \) and \( k_1\ell_2m_3\sigma_3=-m_3(e_{56}+e_{78}) \). Consequently,
  \( \sigma_1 \), \( \sigma_2 \) and \( \sigma_3 \) are non-degenerate
  symplectic forms defining the same orientation.
\end{proof}

The symplectic triple \( (\sigma_1,\sigma_2,\sigma_3) \) on \( M \)
defines a matrix \( Q=(q_{ij}) \) given by \(
\sigma_i\wedge\sigma_j=2q_{ij}\vol_M \), where \( \vol_M \) is the
induced volume form on \( M \).

\begin{proposition}
  The matrices \( G \) and \( Q \) are related via \( G^{-1}=h^2Q \).
  In particular, \(
  \vol_M=\tfrac{h^2}6\sum_{i,j=1}^3g_{ij}\sigma_i\wedge\sigma_j \).
  Moreover, for any positive smooth function \( \lambda \) on \( M \),
  the redefinitions \( \widetilde{Q}=\lambda^2Q \), \(
  \widetilde{G}=\lambda G \), \(
  \widetilde{h}^2=\det(\widetilde{G}^{-1}) \) retain the relation \(
  \widetilde{G}^{-1}=\widetilde{h}^2\widetilde{Q} \).
\end{proposition}
\begin{proof}
  Working locally at a point and using the \( T^3 \)-action, as in the
  proof of Proposition \ref{prop:Phi}, we have
  \begin{gather*}
    \sigma_1\wedge\sigma_2=2\tfrac{k_2\ell_2+k_3\ell_3}{h^2}\vol_M,\quad \sigma_1\wedge\sigma_3=2\tfrac{k_3m_3}{h^2}\vol_M,\quad\sigma_2\wedge\sigma_3=2\tfrac{\ell_3m_3}{h^2}\vol_M,\\
    \tfrac{h^2}{(k_1^2+k_2^2+k_3^2)}\sigma_1^2=\tfrac{h^2}{(\ell_2^2+\ell_3^2)}\sigma_2^2=\tfrac{h^2}{m_3^2}\sigma_3^2=2\vol.
  \end{gather*}
  where \( \vol_M=e_{5678} \) is induced volume form on \( M \). The
  relation between \( Q \) and \( G^{-1} \) now follows directly from
  the expression \eqref{eq:Ginv}, and it immediately implies the last
  two assertions of the proposition.
\end{proof}
As we shall see below, the above behaviour of \( G \) and \( Q \) with
respect to rescaling plays a subtle role in the description of induced
geometry on the hypersurface \( \mathcal X_t \).

It is well-known, cf. \cite{Cabrera:hypersurfaces}, that any
orientable hypersurface in a \( \Spin(7) \)-manifold carries an
induced \( G_2 \)-structure. To express the \( G_2 \)-structure \(
\phi=N\hook\Phi \) on \( \mathcal X_t \) it is useful to rewrite \(
\Phi \) in a way that abuses notation slightly, namely using the forms
defined on \( M \).
\begin{equation}
  \label{eq:altSp7form}
  \begin{split}
    \Phi&=d\nu\wedge\left(\theta_3\wedge\theta_2\wedge\theta_1+\Theta_1\wedge\sigma_1+\Theta_2\wedge\sigma_2+\Theta_3\wedge\sigma_3\right)\eqbreak
    +\theta_3\wedge\theta_2\wedge\sigma_1+\theta_1\wedge\theta_3\wedge\sigma_2+\theta_2\wedge\theta_1\wedge\sigma_3
    + \vol_M.
  \end{split}
\end{equation}
From \eqref{eq:altSp7form} we see that
\begin{equation}
  \label{eq:G23form}
  h\phi=\theta_3\wedge\theta_2\wedge\theta_1+\Theta_1\wedge\sigma_1+\Theta_2\wedge\sigma_2+\Theta_3\wedge\sigma_3.
\end{equation}
Alternatively we may, up to orientation, specify the \( G_2
\)-structure by the four-form \( \psi=\iota^*\Phi\, (=\Hodge\phi) \):
\begin{equation*}
  \psi=\theta_3\wedge\theta_2\wedge\sigma_1+\theta_1\wedge\theta_3\wedge\sigma_2+\theta_2\wedge\theta_1\wedge\sigma_3 +\vol_M.
\end{equation*}
As the \( \Spin(7) \)-structure is torsion-free, the induced
real-analytic \( G_2 \)-structure on \( \mathcal X_t \) is
\emph{cosymplectic}, meaning \( d\psi=0 \).

It turns out that there is a family of smooth cosymplectic \( G_2
\)-structures on \( \mathcal X_t \) obtained by scaling of the volume
form on \( M \):

\begin{proposition}
  \label{prop:coG2fam}
  Let \( (\phi,\psi) \) be the \( G_2 \)-structure on \( \mathcal X_t
  \) described above. For any positive smooth function \( \lambda \)
  on \( M \), the changes \( \lambda^2 Q=:\widetilde Q \) and \(
  \lambda G=:\widetilde G \) of \( Q \) and \( G \), respectively,
  give a new cosymplectic \( G_2 \)-structure \(
  (\widetilde{\phi},\widetilde{\psi}) \) on \( \mathcal X_t \):
  \begin{gather}
    \widetilde{h}\widetilde\phi=\theta_3\wedge\theta_2\wedge\theta_1+\widetilde{\Theta}_1\wedge\sigma_1+\widetilde{\Theta}_2\wedge\sigma_2+\widetilde{\Theta}_3\wedge\sigma_3,\\
    \widetilde\psi=\theta_3\wedge\theta_2\wedge\sigma_1+\theta_1\wedge\theta_3\wedge\sigma_2+\theta_2\wedge\theta_1\wedge\sigma_3
    +\widetilde{\vol}_M,
  \end{gather}
  where \( \widetilde{h}=\det(\widetilde
  Q)^{-\tfrac14}=\lambda^{-\tfrac32}h \), \(
  \widetilde{\Theta}_i=\sum_{j=1}^3\widetilde{q}^{ij}\theta_j=\lambda^{-2}\Theta_i
  \), \(
  \widetilde{\vol}_M\allowbreak=\tfrac16\sum_{i,j=1}^3\allowbreak\tilde{q}^{ij}\sigma_i\wedge\sigma_j\allowbreak=\lambda^{-2}\vol_M
  \).
\end{proposition}
\begin{proof}
  Working locally at a point, as in the proof of Proposition
  \ref{prop:Phi}, we have the basis \(
  \base{e_1,\ldots,\widehat{e_4},\ldots,e_8} \) for \( T^*\mathcal X_t
  \). We now define a new basis \(
  \base{f_1,\ldots,\widehat{f_4},\ldots,f_8} \) for \( T^*\mathcal X_t
  \) by letting \( f_i:=\sqrt{\lambda}e_i \), for \( i=1,2,3 \), and
  \( f_i:=\tfrac1{\sqrt{\lambda}}e_i \), for \( i=5,\ldots,8 \).
  Writing \( \widetilde\phi \) and \( \widetilde{\psi} \) in terms of
  \( f_i \) we have that
  \begin{gather*}
    \widetilde{\phi}=-f_{123}-f_{3}(f_{56}+f_{78})+f_2(f_{57}-f_{68})+f_1(f_{58}+f_{67}),\\
    \widetilde{\psi}=f_{12}(f_{56}+f_{78})+f_{13}(f_{57}-f_{68})-f_{23}(f_{58}+f_{67})+f_{5678},
  \end{gather*}
  which shows that \( \widetilde\phi \) and \( \widetilde{\psi} \)
  define a \( G_2 \)-structure with volume form \(
  \widetilde{\vol}_{\mathcal X}=\tfrac1{\sqrt{\lambda}}\vol_{\mathcal
    X} \). Clearly, \( \widetilde\psi \) is closed. Hence the new \(
  G_2 \)-structure is also cosymplectic.
\end{proof}

\section{Inversion via a flow}
\label{sec:Sp7inv}

We now consider how the reduction procedure from the previous section
may be inverted, constructing a \( \Spin(7) \)-metric starting from a
triple of symplectic forms on a four-manifold \( M \). First we need a
weakening of the notion of coherent symplectic triple \cite[Definition
6.4]{Madsen-S:multi-moment}.

\begin{definition}
  A \emph{weakly coherent symplectic triple \( \CST \)} on a
  four-manifold \( M \) consists of three symplectic forms \( \sigma_1
  \), \( \sigma_2 \), \( \sigma_3 \) that pointwise span a maximal
  positive subspace of \( \Lambda^2T^*M \).
\end{definition}

As in \cite{Donaldson-K:4}, the positive three-dimensional subbundle
\( \Lambda^+=\Span{\sigma_1,\sigma_2,\sigma_3}\subset\Lambda^2T^*M \)
corresponds to a unique oriented conformal structure on \( M \). Fix a
volume form \( \vol_M \) on \( M \) compatible with the orientation
and define a \( 3\times 3\)-matrix \( Q=(q_{ij}) \) by \(
\sigma_i\wedge\sigma_j=2q_{ij}\vol_M \), for \( i,j=1,2,3 \).
Subsequently, denote by \( h \) the positive smooth function
satisfying \( h^{-4}=\det(Q) \). We now consider a \( T^3 \)-bundle \(
\pi_M\colon\,\mathcal X \to M \) endowed with connection one-form \(
\theta=(\theta_1,\theta_2,\theta_3)\in\Omega^1(\mathcal X,\bR^3) \).
We define three one-forms \( \Theta_i \), for \( i=1,2,3 \), by the
formula \( \Theta_i=\sum_{j=1}^3q^{ij}\theta_j \). Finally, denote the
curvature by \( F=\pi^*_M(d\theta)\in\Omega^2(M,\bR^3) \). With these
definitions in mind we have:

\begin{proposition}
  \label{prop:cosympl}
  Let \( (M,\CST) \) be a weakly coherent tri-symplectic
  four-manifold. Suppose that \( \mathcal X \) is a principal \( T^3
  \)-bundle over \( M \) with connection one-form \(
  \theta=(\theta_1,\theta_2,\theta_3) \) and curvature \( F \). Define
  a three-form \( \phi \) and a four-form \( \psi \) by
  \begin{equation}
    \label{eq:cosymplform}
    \begin{split}
      h\phi&=\theta_3\wedge\theta_2\wedge\theta_1+\Theta_1\wedge\sigma_1+\Theta_2\wedge\sigma_2+\Theta_3\wedge\sigma_3,\\
      \psi&=\theta_3\wedge\theta_2\wedge\sigma_1+\theta_1\wedge\theta_3\wedge\sigma_2+\theta_2\wedge\theta_1\wedge\sigma_3
      +\vol_M.
    \end{split}
  \end{equation}
  Then \( \phi \) determines a \( G_2 \)-structure on \( \mathcal X \)
  satisfying \( \Hodge\phi=\psi \).

  Let \( A=(a_{ij}) \) be the \( 3\times 3 \)-matrix defined pointwise
  by the projection \( F^+=(\sigma_1,\sigma_2,\sigma_3)A \). Then the
  \( G_2 \)-structure \( \phi \) is cosymplectic if and only if the
  matrix \( QA \) is symmetric:
  \begin{equation}
    \label{eq:curvcond}
    QA = A^tQ
  \end{equation}
\end{proposition}
\begin{proof}
  Write the entries of \( G^{-1}:=h^2Q \) as in \eqref{eq:Ginv} and
  then express the functions \( k_1,\ldots,m_3 \) in terms of the
  entries \( g^{ij} \) of \( G^{-1}=h^2Q \). Next, choose a conformal
  basis \( e_5,e_6,e_7,e_8 \) of \( T^*M \) so that \( h\sigma_i \)
  are as in the proof of Proposition \ref{prop:Phi} and then write \(
  \theta_1=k_1e_1+k_2e_2+k_3e_3 \), \( \theta_2=\ell_2e_2+\ell_3e_3
  \), \( \theta_3=m_3e_3 \). It now follows, using Proposition
  \ref{prop:coG2fam}, that the basis \(
  \base{e_1,\ldots,\widehat{e_4},\ldots,e_8} \) is a \( G_2 \)-basis
  for \( T^*\mathcal X \) with defining form \( \phi \) given via
  \eqref{eq:cosymplform}.

  For the final assertion we need to study the condition \( d\psi=0
  \). The equation \( d\psi=0 \) holds if and only if one has
  \begin{equation*}
    d\theta_1\wedge\sigma_2-d\theta_2\wedge\sigma_1=d\theta_3\wedge\sigma_1-d\theta_1\wedge\sigma_3=d\theta_2\wedge\sigma_3-d\theta_3\wedge\sigma_2=0.
  \end{equation*}
  A calculation shows that these relations correspond to the three
  equations
  \begin{equation}
    \begin{split}
      -a_{13}q_{12}+a_{12}q_{13}-a_{23}q_{22}+(a_{22}-a_{33})q_{23}+a_{32}q_{33}&=0,\\
      a_{13}q_{11}+a_{23}q_{12}+(a_{33}-a_{11})q_{13}-a_{21}q_{23}-a_{31}q_{33}&=0,\\
      -a_{12}q_{11}+(a_{11}-a_{22})q_{12}-a_{32}q_{13}+a_{21}q_{22}+a_{31}q_{23}&=0,
    \end{split}
  \end{equation}
  and these are equivalent to the condition \eqref{eq:curvcond}.
\end{proof}

\begin{remark}
  Condition \eqref{eq:curvcond} on \( F \) is independent of the
  choice of orientation compatible volume form on \( M \). Though the
  bilinear form on \( \Lambda^2T^*M \), given by wedging, is only
  well-defined after choosing a representative volume form,
  self-adjointness of the projection \( F^+ \in \Lambda^+\subset
  \Lambda^2T^*M \) does not depend on the specific choice.
  
  Provided the assumptions of Proposition \ref{prop:cosympl} hold, we
  therefore obtain a family of cosymplectic \( G_2 \)-manifolds. This
  is a consequence of Proposition \ref{prop:coG2fam}, and contrasts
  with the corresponding analysis of \( \SU(3) \)-structures on \( T^2
  \)-bundles over coherently tri-symplectic four-manifolds
  \cite[Proposition 6.5]{Madsen-S:multi-moment}. In that situation we
  made a particular choice of volume form to obtain a half-flat
  structure.
\end{remark}
\begin{remark}
  Existence of three-torus bundles over a weakly coherent
  tri-symplectic four-manifold \( (M,\CST) \) is related to Chern-Weil
  theory. One finds that for any closed two-form \( F \) with integral
  periods, \( F\in\Omega^2_{\bZ}(M,\bR^3) \), there exists a \( T^3
  \)-bundle \( \pi_M\colon\,\mathcal X\to M \) with connection
  one-form \( \theta \) that satisfies \( \pi_M^*(d\theta)=F \).
\end{remark}

Studying a certain Hamiltonian flow, Hitchin \cite{Hitchin:forms}
developed a relationship between torsion-free \( \Spin(7) \)-metrics
and cosymplectic \( G_2 \)-manifolds. In particular, he derived
evolution equations that describe the one\bdash{}dimensional flow of a
cosymplectic \( G_2 \)-manifold along its unit normal in a
torsion-free \( \Spin(7) \)-manifold. In inverting our construction,
one could use Hitchin's flow on the cosymplectic structure of
Proposition \ref{prop:cosympl}. However, Hitchin's flow does not
preserve the level sets of the multi-moment map: the unit normal is \(
h(d\nu)^\sharp \), but \( \partial/{\partial\nu}=h^2(d\nu)^\sharp \).
It is therefore more natural for us to determine the flow equations
associated to the latter vector field.

\begin{proposition}
  \label{prop:mflow}
  Suppose \( T^3 \) acts freely on a connected eight-manifold \( Y \)
  preserving the torsion-free \( \Spin(7) \)-structure \( \Phi \) and
  admitting a multi-moment map \( \nu \). Let \( M \) be the
  topological reduction \( \nu^{-1}(t)/{T^3} \) for any \( t \) in the
  image of \( \nu \). Then \( M \) is equipped with a \( t
  \)-dependent weakly coherent real-analytic symplectic triple \(
  \sigma_1 \), \( \sigma_2 \), \( \sigma_3 \) and the seven-manifold
  \( \mathcal X_t=\nu^{-1}(t) \) carries a cosymplectic real-analytic
  \( G_2 \)-structure of the form \eqref{eq:cosymplform} . On \(
  \mathcal X_t \) the following evolution equation holds:
  \begin{equation}
    \label{eq:evol}
    \psi'=d(h\phi),
  \end{equation}
  where \( ' \) denotes differentiation with respect to \( t \).

  Conversely, given a cosymplectic real-analytic \( G_2 \)-structure
  of the form \eqref{eq:cosymplform} defined on a seven-manifold \(
  \mathcal X_0 \). Then the flow equation \eqref{eq:evol} admits a
  unique solution on some open neighbourhood of \( \mathcal
  X_0\times\{ 0\} \subset \mathcal X_0\times \bR \), and that solution
  determines a torsion-free \( \Spin(7) \)-structure.
\end{proposition}
\begin{proof}
  We have
  \begin{equation*}
    \Phi= hd\nu\wedge\phi +\psi.
  \end{equation*}
  This has derivative
  \begin{equation*}
    d\Phi=d\nu\wedge(-dh\wedge \phi-hd\phi)+d\psi.
  \end{equation*}
  By assumption, the \( G_2 \)-structure is cosymplectic, i.e., \(
  d\psi=0 \) on each level set. We therefore find that \( d\Phi=0 \)
  if and only if
  \begin{equation*}
    0=\frac{\partial}{\partial\nu}\hook d\Phi=-d(h\phi)+\psi'.
  \end{equation*}
  Hence we have a torsion-free \( \Spin(7) \)-structure if and only if
  the evolution equation \eqref{eq:evol} is satisfied.

  Observe that equation \eqref{eq:evol} together with an initial
  cosymplectic \( G_2 \)-structure on \( \mathcal X_0 \) already
  ensure that the family consists of cosymplectic structures; the time
  derivative of \( d\psi \) vanishes according to \eqref{eq:evol}.
  
  We note that given real-analytic initial data, the
  Cauchy-Kovalevskaya Theorem applies. Therefore we obtain existence
  and uniqueness of a solution defined on some open neighbourhood of
  \( \mathcal X_0\times \{0\} \subset \mathcal X_0\times \bR\).

  For later use, we shall rewrite the evolution equation as a set of
  first order differential equations for the quantities defined by
  data on \( M \). First we note that
  \begin{equation*}
    \begin{split}
      \psi'&=\sigma'_1\wedge\theta_3\wedge\theta_2+\sigma'_2\wedge\theta_1\wedge\theta_3+\sigma'_3\wedge\theta_2\wedge\theta_1
      +(\theta'_2\wedge\sigma_3-\theta'_3\wedge\sigma_2)\wedge\theta_1\eqbreak+(\theta'_3\wedge\sigma_1-\theta'_1\wedge\sigma_3)\wedge\theta_2
      +
      (\theta'_1\wedge\sigma_2-\theta'_2\wedge\sigma_1)\wedge\theta_3+\vol'_M.
    \end{split}
  \end{equation*}
  \begin{equation*}
    \begin{split}
      d(h\phi)&=d\theta_1\wedge\theta_3\wedge\theta_2+d\theta_2\wedge\theta_1\wedge\theta_3+d\theta_3\wedge\theta_2\wedge\theta_1
      \eqbreak+\sigma_1\wedge d\Theta_1 +\sigma_2\wedge
      d\Theta_2+\sigma_3\wedge d\Theta_3,
    \end{split}
  \end{equation*}
  where
  \begin{equation*}
    \sum_{i=1}^3\sigma_i\wedge d\Theta_i=\sum_{i,j=1}^3\sigma_i\wedge\left(d(q^{ij})\wedge\theta_j+q^{ij}d\theta_j\right).
  \end{equation*}
  From these equations we get the \( t \)-derivatives for \( \sigma_1
  \), \( \sigma_2 \), \( \sigma_3 \):
  \begin{equation}
    \label{eq:tdersigma}
    \sigma'_i=d\theta_i,\quad\textrm{for}\quad i=1,2,3.
  \end{equation}
  The \( t \)-derivative of the connection one-form \(
  \theta=(\theta_1,\theta_2,\theta_3) \) is given by
  \begin{equation}
    \label{eq:tdertheta}
    \theta'_i\wedge\sigma_j-\theta'_j\wedge\sigma_i=\sum_{\ell=1}^3\sigma_\ell\wedge dq^{\ell k},\quad \textrm{for}\quad \sgn(ijk)=+1.
  \end{equation}
  The volume form \( \vol_M \) evolves via
  \begin{equation}
    \label{eq:tdervolM}
    \vol'_M=\sum_{i,j=1}^3q^{ij}\sigma_i\wedge d\theta_j.
  \end{equation}
  Finally the \( t \)-derivatives of entries \( q_{ij} \) of \( Q \)
  may be expressed via
  \begin{equation}
    \label{eq:tderq}
    2q'_{ij}\vol_M=d\theta_i\wedge\sigma_j+\sigma_i\wedge d\theta_j-2q_{ij}\sum_{k,\ell=1}^3q^{k\ell}\sigma_k\wedge d\theta_\ell ,\quad\textrm{for}\quad i,j=1,2,3.
  \end{equation}
  Note that the equations for the entries \( q_{ij} \) now determine
  the evolution of \( h \) and \( G \) via the relations \(
  h^{-4}=\det(Q) \) and \( G^{-1}=h^2Q \), respectively.
\end{proof}

\begin{remark}
  By solving the flow equations we obtain a holonomy \( \Spin(7)
  \)-metric with three-torus symmetry. Indeed, if \( g_M \) is the
  time-dependent metric in the conformal class on \( M \) with volume
  form \( \vol_M \), then the \( \Spin(7) \)-metric is explicitly
  \begin{equation}
    \label{eq:Sp7metr}
    \begin{split}
      &h^2dt^2+g_M+g_{11}\theta_1^2+g_{22}\theta_2^2+g_{33}\theta_3^2+g_{12}\theta_1\theta_2
      +g_{13}\theta_1\theta_3+g_{23}\theta_2\theta_3,
    \end{split}
  \end{equation}
  where \( G=(g_{ij})=h^{-2}Q^{-1} \).
  
  Real-analyticity of the cosymplectic \( G_2 \)-structures is a
  subtle matter. Bryant's study of the Hitchin flow
  \cite{Bryant:nonemb} shows that non-analytic initial half-flat \(
  \SU(3) \)-structures can lead to an ill-posed Hitchin system that
  has no solution.
\end{remark}

\begin{remark}
  Though the torsion-free \( G_2 \)-manifolds studied in
  \cite{Madsen-S:multi-moment} fiber over (weakly) coherently
  tri-symplectic four-manifolds, they do not fit naturally into the
  above framework. The constructed \( G_2 \)-flow does not preserve
  the \( \Spin(7) \)-data.
\end{remark}

Summarising the results discussed so far we have:
           
\begin{theorem}
  Let \( (Y^8,\Phi) \) be a torsion-free \( \Spin(7) \)-manifold with
  a free \( T^3 \) symmetry and admitting a multi-moment map. Then
  the reduction \( M \) at level \( t \) carries a weakly coherent
  real-analytic symplectic triple and the level set \( \mathcal X_t \)
  is the total space of a \( T^3 \)-bundle over \( M \) satisfying
  condition \eqref{eq:curvcond} on the curvature.

  Conversely, let \( (M,\CST) \) be a weakly coherent tri-symplectic
  four-manifold with a closed two-form \( F \in\Omega^2_\bZ(M,\bR^3)
  \) and a choice of orientation compatible volume form. Assume \( F
  \) satisfies condition \eqref{eq:curvcond}. When these data are
  real-analytic, they define a torsion-free \( \Spin(7) \)-metric with
  \( T^3 \)-symmetry.  \qed
\end{theorem}

\section{Examples}
\label{sec:ex}

Let us now turn to some examples that illustrate the analysis of the
previous two sections. First we show that even in the flat case \(
\bR^8 \), with \( T^3\subset\SU(4) \), the geometry of the reduction
procedure is somewhat complicated. Thereafter we study hyperK{\"a}hler
four-manifolds, complementing previous examples that have appeared in
the physics literature
\cite{Gibbons-LPS:domain-walls,Santillan-Gi:toric}.

\begin{example}[The flat model \( \bR^8 \)]

  Consider \( Y=\bR^8=\bC^4 \) endowed with the usual four-form and
  the action of the standard diagonal maximal torus \(
  T^3\subset\SU(4) \). Concretely, \( \Phi \) is given by
  \begin{equation*}
    \begin{split}
      \Phi&=\tfrac12\left(\tfrac{i}{2}(dz_1\wedge
        d\bar{z}_1+dz_2\wedge d\bar{z}_2+dz_3\wedge
        d\bar{z}_3+dz_4\wedge
        d\bar{z}_4)\right)^2\eqbreak+\re(dz_1\wedge dz_2 \wedge dz_3
      \wedge dz_4),
    \end{split}
  \end{equation*}
  and \( T^3 \) acts by \(
  (e^{i\theta_1},e^{i\theta_2},e^{i\theta_3})\cdot(z_1,z_2,z_3,z_4)=(e^{i\theta_1}z_1,e^{i\theta_2}z_2,e^{i\theta_3}z_3,e^{-i(\theta_1+\theta_2+\theta_3)}z_4)
  \). The action is generated by the vector fields \(
  U_j=\re\left\{i(z_j\tfrac{\partial}{\partial
      z_j}-z_4\tfrac{\partial}{\partial z_4})\right\} \), for \(
  j=1,2,3 \). It follows that a multi-moment map \( \nu\colon\,Y\to\bR
  \) is given by
  \begin{equation*}
    \nu(z_1,z_2,z_3,z_4)=\tfrac18\im(z_1z_2z_3z_4).
  \end{equation*}
  By definition, the \( T^3 \) reduction of \( Y \) at level \( t \)
  is the quotient space \( M_t={\nu^{-1}(t)}/{T^3} \). In this case \(
  M_0 \) is singular, whereas \( M_t \) is a smooth manifold for each
  \( t\ne0 \). Indeed, considering \( \Xi_t\colon\,M_t\to\bR^4 \)
  given by
  \begin{equation*}
    \begin{split}
      \Xi
      _t(z_1,z_2,&z_3,z_4)=\left(\tfrac{\norm{z_1}^2{-}\norm{z_4}^2}2,\tfrac{\norm{z_2}^2{-}\norm{z_4}^2}2,\tfrac{\norm{z_3}^2{-}\norm{z_4}^2}2,\re(z_1z_2z_3z_4)\right)\eqbreak
      =:(v_1,v_2,v_3,w),
    \end{split}
  \end{equation*}
  we have global smooth coordinates on \( M_t \) for \( t\ne0 \).

  In this smooth case, writing \(
  (\eta_1,\eta_2,\eta_3)=(dv_1,dv_2,dv_3)G^{-1} \), the two-forms \(
  \sigma_1 \), \( \sigma_2 \), \( \sigma_3 \) are given by
  \begin{gather*}
    16\sigma_1=\eta_1\wedge dw + 4dv_2\wedge dv_3,\quad 16\sigma_2=\eta_2\wedge dw + 4dv_3\wedge dv_1,\\
    16\sigma_3=\eta_3\wedge dw + 4dv_1\wedge dv_2.
  \end{gather*}
  These forms depend (implicitly) on \( t \) via the relations \(
  4g_{ij}=\delta_{ij}\norm{z_i}^2+\norm{z_4}^2 \), for \( i,j=1,2,3
  \), and \( z_1z_2z_3z_4=w+8it \). In particular \( g_{ij} \) is a
  non-constant positive function \( f \), for \( i\neq j \). Thus the
  weakly coherent triple does not specify a coherent triple, in
  particular it is not a hyperK{\"a}hler structure.
  
  The (oriented) conformal class has representative metric
  \begin{equation*}
    \begin{split}
      &
      \frac{h^2}{16}dw^2+g_{11}\eta_1^2+g_{22}\eta_2^2+g_{33}\eta_3^2+f(\eta_1\eta_2+\eta_1\eta_3+\eta_2\eta_3),
    \end{split}
  \end{equation*}
  where \( h^2=\det(G^{-1}) \).

  The curvature \( F = (F_1,F_2,F_3) \) of the principal \( T^3
  \)-bundle \( \nu^{-1}(t)\to M_t \) is given by
  \begin{equation*}
    \begin{split}
      F_1&=2th^2\eta_w\wedge((2g_{22}g_{33}-f(g_{22}+g_{33}))\eta_1\eqbreak+(g_{33}-f)(g_{22}-2f)\eta_2+(g_{22}-f)(g_{33}-2f)\eta_3),\\
      F_2&=2th^2\eta_w\wedge((2g_{11}g_{33}-f(g_{11}+g_{33}))\eta_2\eqbreak+(g_{11}-f)(g_{33}-2f)\eta_3+(g_{33}-f)(g_{11}-2f)\eta_1),\\
      F_3&=2th^2\eta_w\wedge((2g_{11}g_{22}-f(g_{11}+g_{22}))\eta_3\eqbreak+(g_{22}-f)(g_{11}-2f)\eta_1+(g_{11}-f)(g_{22}-2f)\eta_2).
    \end{split}
  \end{equation*}
  where \( \eta_w=g_{ww}^{-1}\,dw \) satisfies \(
  \eta_w((dw)^\sharp)=1 \) and \( \eta_w((dv_i)^\sharp)=0 \), for \(
  i=1,2,3 \). Note that \( F \neq F^+ \).

  In the singular case \( t=0 \), the three-torus collapses in three
  different ways: to a point, a circle or a two-torus. At the origin
  \( (z_1,z_2,z_3,z_4)=0 \) the three-torus collapses to a point.
  Next, if \( z_i=z_j=z_k=0 \) for exactly three different indices,
  then the torus collapses to a circle. In terms of the quadruple \(
  (v_1,v_2,v_3,w) \) this collapsing happens for \( w=0 \) when \( v_1
  \), \( v_2 \), \( v_3 \) satisfy one of the following constraints:
  \( (v_1=v_2=v_3\leq0) \), \( (v_1=v_2=0,\,v_3\geq0) \), \(
  (v_1=v_3=0,\,v_2\geq0) \) or \( (v_2=v_3=0,\,v_1\geq0) \). Finally,
  if \( z_i=z_j=0 \) for exactly two different indices, the \( T^3 \)
  collapses to a two-torus. This happens for \( w=0 \) when \( v_1 \),
  \( v_2 \), \( v_3 \) satisfy one of: \( (v_1=v_2\leq0) \), \(
  (v_1=v_3\leq0) \), \( (v_1=0,\,v_2,v_3\geq0) \), \( (v_2=v_3\leq0)
  \), \( (v_2=0,\,v_1,v_3\geq0) \) or \( (v_3=0,\,v_1,v_2\geq0) \).
\end{example}

\begin{example}[Examples from hyperK{\"a}hler manifolds]
  \label{ex:K3}
  Let \( M \) be a hyperK\"{a}hler four-manifold. Then \( M \) comes
  equipped with three symplectic forms \( \sigma_1 \), \( \sigma_2 \),
  \( \sigma_3 \) that satisfy the relations \(
  \sigma_i\wedge\sigma_j=\delta_{ij}\sigma_k^2\) for \( i,j,k=1,2,3
  \). Choosing the volume form \( \vol^0_M = \tfrac12\sigma_1^2\), we
  have that \( Q=\diag(1,1,1) \). The compatible hyperK{\"a}hler
  metric is denoted by \( g_M^0 \).

  Let \( \sigma=(\sigma_1,\sigma_2,\sigma_3) \) denote the
  hyperK{\"a}hler triple and assume there is a constant matrix \( A =
  (a_{ij}) \) such that \( \sigma A\in \Omega^2_{\bZ}(M,\bR^3) \).
  Then we may construct a \( T^3 \)-bundle over \( M \) with
  connection one-form \( \theta \) that has curvature \( F =\sigma A
  \). The total space \( \mathcal X_0 \) of this bundle carries the \(
  G_2 \)-structure of Proposition \ref{prop:cosympl}, which is now
  cosymplectic if and only if \( A \) is symmetric. The associated
  metric on \( \mathcal X_0 \) is complete if the hyperK{\"a}hler base
  manifold is complete.

  We shall illustrate how one may solve the flow equations
  \eqref{eq:tdersigma}--\eqref{eq:tderq} starting from the above data
  at \( t=0 \). As an a priori simplifying assumption we consider the
  case when \( F'=0 \), i.e., the curvature is \( t \)-independent.
  Then the differential equations for the symplectic triple simply
  reads \( \sigma'= \Omega A \), where \( \Omega=\sigma(0) \).
  Integrating, we find that \( \sigma(t)=\Omega(1+tA) \).

  We next solve the equations \eqref{eq:tdervolM} and
  \eqref{eq:tderq}. First we observe that the volume develops
  according to the equation \( \vol'_M=v\vol_M^0 \), where \( v =
  2\Tr(Q^{-1}(1+tA)A) \). We may therefore write \(
  \vol_M(t)=V(t)\vol_M^0 \), where \( V'=v \) and \( V(0)=1 \). The
  equation for \( Q' \) now takes the form \( VQ'=2(1+At)A-vQ \). It
  follows that we must find the unique solution of the differential
  equation \( \ln(V)'=2Tr((1+tA)^{-1}A) \), \( V(0)=1 \). We find that
  \( V(t) = \det(1+tA)^2 \). Consequently, \( \vol_M \) and \( Q \)
  take the form \( \vol_M(t)=\det(1+tA)^2\vol^0_M \) and \(
  \det(1+tA)^2Q(t)=(1+tA)^2 \). Note also that \( h(t)=\det(1+tA) \)
  and that \( dq_{ij}(t)=0 \). The latter observation implies, by
  \eqref{eq:tdertheta}, that the connection one-form is \( t
  \)-independent, \( \theta(t)=\theta \).

  The above solution is defined on \( \mathcal X_0\times I \), where
  the interval \( I \subset \bR \) is determined by non-degeneracy of
  the matrix \( 1+tA \) and \( 0 \in I \). By uniqueness of the
  solution on \( \mathcal X_0\times I \), we deduce that the condition
  \( F'=0 \) already follows from the initial data, i.e., it is not a
  simplifying assumption.

  The torsion-free \( \Spin(7) \)-structure corresponding to the above
  solution has associated metric \( g \) given by
  \begin{equation}
    \begin{split}
      h^2(t)dt^2+h(t)g^0_M&+h(t)^{-2}\left(\sum^3_{i=1}q^{ii}(t)\theta_i^2+\sum_{1\leq
          i<j\leq 3}q^{ij}(t)\theta_i\theta_j\right).
    \end{split}
  \end{equation}
  If the initial hyperK{\"a}hler four-manifold is complete, we may
  describe completeness properties of \( g \) in terms of the matrix
  \( A \). Provided \( g \) remains finite and non-degenerate,
  completeness corresponds to completeness of \( h(t)^2dt^2 \) on \( I
  \), cf. \cite{Bishop-O:warped}. We now find that \( g \) is
  half-complete, cf. \cite{Apostolov-S:K-G2}, if and only if \( A \)
  does not have two eigenvalues of opposite sign; the metric is
  complete only for \( A=0 \).
\end{example}

\providecommand{\bysame}{\leavevmode\hbox to3em{\hrulefill}\thinspace}
\providecommand{\MR}{\relax\ifhmode\unskip\space\fi MR }
\providecommand{\MRhref}[2]{%
  \href{http://www.ams.org/mathscinet-getitem?mr=#1}{#2} }
\providecommand{\href}[2]{#2}

\begin{small}
  \parindent0pt\parskip\baselineskip T.B.Madsen

  Department of Mathematics and Computer Science, University of
  Southern Denmark, Campusvej 55, DK-5230 Odense M, Denmark
  
  \textit{and}

  CP\textsuperscript3-Origins, Centre of Excellence for Particle
  Physics Phenomenology, University of Southern Denmark, Campusvej 55,
  DK-5230 Odense M, Denmark

  \textit{E-mail}: \url{tbmadsen@imada.sdu.dk}
\end{small}

\end{document}